\documentclass{article}[12pt]
\usepackage[utf8]{inputenc}
\usepackage{authblk}
\usepackage{amsmath, amssymb, amsthm, amsfonts, gensymb, tikz, commath, float, mathtools, enumerate, breqn, circuitikz}
\usetikzlibrary{positioning, arrows, arrows.meta, cd}
\usetikzlibrary{shapes,backgrounds,positioning,petri,topaths,calc}

\usepackage[all]{xy}
\usepackage{tabularx}
\usepackage{multirow}
\usepackage{esvect}
\usepackage{subcaption}
\usepackage{stmaryrd} 

\usepackage[text={15cm,22cm}]{geometry} 
\usepackage[colorlinks=true,%
            linkcolor=red!50!black,%
            citecolor=blue!50!black,%
            urlcolor=darkgray]{hyperref}  
\theoremstyle{plain}

\newtheorem{lem}{Lemma}[section]
\newtheorem{thm}{Theorem}
\newtheorem*{thmE}{Euclide's Theorem}
\newtheorem{cor}[lem]{Corollary}
\newtheorem{prop}[lem]{Proposition}
\newtheorem*{conj}{Conjecture}

\theoremstyle{definition}
\newtheorem*{rem}{Remark}
\newtheorem{ex}[lem]{Example}

\newtheorem{defn}[lem]{Definition}

\newcommand{\Z}{\mathbb{Z}}

\newcommand{\Q}{\mathbb{Q}}

\newcommand{\A}{\mathcal{A}}
\newcommand{\B}{\mathcal{B}}
\newcommand{\cC}{\mathcal{C}}
\newcommand{\cD}{\mathcal{D}}

\newcommand{\cN}{\mathcal{N}}

\newcommand{\Tr}{\textup{Tr}}

\newcommand{\SL}{\mathrm{SL}}
\newcommand{\PSL}{\mathrm{PSL}}

\newcommand{\half}{\frac{1}{2}}

\def\a{\alpha}
\def\b{\beta}
\def\d{\delta}

\def\g{\gamma}

\title{Quantizing Pythagorean triples}

\author{
Hugo Mathevet,
Sophie Morier-Genoud, 
Valentin Ovsienko}

\date{}

\begin{document}

\maketitle

A classical Pythagorean triple $(a,b,c)$ is a triplet
of positive integers $a,b$ and $c$ satisfying the Diophantine equation 
$$
a^2+b^2=c^2
$$
called the Pythagoras equation.
This antique subject has always been and remains an active field of research.
For a detailed account of its historical development,
the reader is invited to consult
Sierpi\'nski's classical book~\cite{Sie}.
A surprising and curious idea of developing the entire number theory through the 
Pythagoras prism is proposed in~\cite{Tak}.

A Pythagorean triple $(a,b,c)$ is called {\it primitive} if $a,b,c$ are coprime,
but we will also be interested in the case where the greater common divisor
$\gcd(a,b,c)$ of $a,b,c$ equals $2$.
We have two possibilities
$$
\gcd(a,b,c)=
\left\{
\begin{array}{l}
1,\\
2.
\end{array}
\right.
$$
When $\gcd(a,b,c)=1$, we will require that $a$ is even, 
when $\gcd(a,b,c)=2$, we will require that $a/2$ is odd.
Such Pythagorean triple $(a,b,c)$ is sometimes called {\it standard}; see, e.g.~\cite{Tra}.

The Euclide formula provides a simple way to associate a Pythagorean triple
with an arbitrary rational number~$\frac{m}{n}\geq1$.
Consider positive coprime integers $m\geq n$, then it is easy to check that
\begin{equation}
\label{EuclForm}
a=2mn,
\qquad
b=m^2-n^2,
\qquad
c=m^2+n^2
\end{equation}
form a Pythagorean triple.
To some extent, the following statement can be attributed to Euclide.

\begin{thmE}
\label{ClassThm}
For every standard Pythagorean triple
there exist coprime positive integers $m,n$ such that 
the triplet $(a,b,c)$ are given by~\eqref{EuclForm}.
\end{thmE}

In other words, standard Pythagorean triples are parametrized by rationals~$\frac{m}{n}>1$.
For instance, 
the first nontrivial Pythagorean triple $(4,3,5)$ corresponds to $\frac{2}{1}$,
the next standard (but not primitive!) triple $(6,8,10)$ corresponds to $\frac{3}{1}$,
then  $\frac{3}{2}$ produces $(12,5,13)$, etc.

The Pythagorean triple~\eqref{EuclForm} is primitive if and only if
$m$ and $n$ are coprime and not both odd.
For the  standard triples, this restriction is removed, that is,
$m$ and $n$ are arbitrary positive coprime integers.
This is an important reason to extend considerations from primitive to standard triples.

\section{The $q$-deformed Pythagoras equation}

The goal of this article is to introduce and study 
a new natural $q$-analogue of the Pythagoras equation.
We consider three polynomials,
$\A,\B,\cC$, in one variable denoted by~$q$ (following a certain tradition).
We say that $\A,\B,\cC$ satisfy the $q$-deformed Pythagoras equation if
\begin{equation}
\label{PythEq}
\A(q)^2+q\B(q)^2=\cC(q)\cC^*(q),
\end{equation}
where $\cC^*$ is the polynomial reciprocal to $\cC$, i.e.
\begin{equation}
\label{RecipEq}
\cC^*(q)=q^{\deg(\cC)}\cC(q^{-1}).
\end{equation}

To give an idea about the polynomials appearing
in our context, consider two elementary examples.
More examples will be given later.

(1)
Our first nontrivial example of solution to~\eqref{PythEq} is
the only solution corresponding to the first nontrivial primitive Pythagorean triple $(a,b,c)=(4,3,5)$
and satisfying the conditions \ref{Con1}--\ref{Con3}:
$$
\A(q)=1+q+q^2+q^3=:\left[4\right]_q,
\qquad\qquad
\B(q)=1+q+q^2=:\left[3\right]_q,
$$
with $\cC$ and $\cC^*$ given by
$$
\cC(q)=1+2q+q^2+q^3
\qquad\quad\hbox{and}\quad\qquad
\cC^*(q)=1+q+2q^2+q^3.
$$
Note that $\cC$ and $\cC^*$ are interchangeable and cannot be distinguished from each other.
Note also that the order matters: $a=4$, and $b=3$, but not vice-versa.
To better understand this example, recall that the polynomial
\begin{equation}
\label{qBn}
\left[n\right]_{q}:=1+q+q^{2}+\cdots+q^{n-1}=\textstyle\frac{1-q^n}{1-q},
\end{equation}
is commonly considered as the $q$-analogue of a (positive) integer~$n$.
Extensively used in such areas as quantum groups, quantum calculus, etc. this
notion goes back to Euler ($\approx$1760) and Gauss ($\approx$1808).

(2)
Our second example is obtained by doubling the previous one,
but the roles of $a$ and $b$ are exchanged: $(a,b,c)=(6,8,10)$.
Our solution to~\eqref{PythEq} in this case is
\begin{eqnarray*}
\A(q) &=& 1+q+q^2+q^3+q^4+q^5=\left[6\right]_q,
\\[4pt]
\B(q) &=& 1+2q+2q^2+2q^3+q^4=(1+q)(1+q^2)^2,
\\[4pt]
\cC(q) &=& 1+q+2q^2+3q^3+2q^4+q^5=(1+2q^2+q^3+q^4)(1+q).
\end{eqnarray*}

We did not find the equation~\eqref{PythEq} in the literature.
Let us mention that more straightforward polynomial generalizations of the Pythagoras equation
such as $\A(q)^2+\B(q)^2=\cC(q)^2$ have been considered by many authors; 
see, e.g.~\cite{DLS,CC}.
This equation is distantly related to the vast subject of sum of squares of polynomials
and Hilbert's seventeenth problem.
However, polynomials satisfying this equation cannot have nice properties 
we will be interested in.

\section{An interesting class of solutions}

We will search for solutions to the equation~\eqref{PythEq}
satisfying the following properties 
which seem quite natural from a combinatorial point of view.

\begin{enumerate}
\item
\label{Con1}
The polynomials $\A,\B,\cC$ have positive integer coefficients;

\item
\label{Con2}
The polynomials $\A$ and $\B$ are self-reciprocal, or ``palindromic'';

\item
\label{Con3}
The polynomials $\A,\B,\cC$ and $\cC^*$ are monic, 
i.e. their leading and lower degree coefficients are equal to~$1$.

\end{enumerate}

The positivity condition \ref{Con1} implies that the triple of integers 
$$
(a,b,c):=(\A(1),\B(1),\cC(1))
$$ 
is a classical Pythagorean triple.
We will therefore say that the triplet of polynomials $(\A,\B,\cC)$
corresponds to this Pythagorean triple
and is its $q$-analogue, or ``quantization''.
This term and notion is extensively used in mathematical physics, but also in combinatorics.
Roughly speaking, quantization means replacing a single quantity by a discrete
(finite) sequence that potentially contain more information.
In our situation, we replace single integers $a,b,c$ by sequences of coefficients
of the polynomials $\A,\B,\cC$, 
and each coefficient of these polynomials must have a meaning.
Another approach to quantization consists in replacing 
commutative algebraic structure by non-commutative, as explored in~\cite{AE}.
Note that two approaches are related, but their comparison is far beyond the scope of this article.

Our main result is the following existence statement.

\begin{thm}
\label{ExtUniq}
For every standard Pythagorean triple $(a,b,c)$
there exists a solution $(\A,\B,\cC)$ to~\eqref{PythEq}
satisfying the conditions \ref{Con1}, \ref{Con2}, and \ref{Con3}
and corresponding to $(a,b,c)$.
\end{thm}

To illustrate this theorem, let us give an infinite
  series of solutions enumerated by integers:
\begin{eqnarray*}
\A_{\frac{n}{1}}(q)&=&
(1+q^n)\left[n\right]_q,
\\[4pt]
\B_{\frac{n}{1}}(q)&=&
\left[n+1\right]_q\left[n-1\right]_q,
\\[4pt]
\cC_{\frac{n}{1}}(q)&=&
1+q\left[n\right]_q^2,
\end{eqnarray*}
where $\left[n\right]_q$ is the $q$-integer; see~\eqref{qBn}.

Our construction of polynomial Pythagorean triples is based on
the $q$-deformed action of the modular group $\PSL(2,\Z)$ on the rational projective line.
It was used in~\cite{MGOfmsigma} to define the notion of $q$-deformed rational numbers.
For more details, see~\cite{LMGadv} and  a survey~\cite{MGOmn}.
Note that our approach is quite close to that of~\cite{EJMGO}
where $q$-deformations of Markov triples were studied with the help of
the $q$-deformed action of the modular group.

Similarly to the classical case,
we obtain an infinite series of $q$-deformed Pythagorean triples
organized in a form of a binary tree.
Replacing the rational numbers with
$q$-rationals, we obtain a $q$-analogue of the Euclid formula.

We conjecture that our solutions have another remarkable property.
A sequence of real numbers is said to be {\it unimodal}
if it increases (not strictly monotonically) to a maximum, then decreases monotonically.
Unimodal sequences have a single peak and no oscillations.

We wish the following additional property
\begin{enumerate}
\item[$4^*$.]
\label{Con4}
The  sequences of coefficients of the polynomials $\A,\B,\cC$ are unimodal,
\end{enumerate}
but we are unable to guarantee it!
The unimodality property is usually difficult to prove.
Note that this property for $q$-deformed rational numbers
was conjectured in~\cite{MGOfmsigma} and proved in~\cite{OgRa}.

\begin{conj}
The  sequences of coefficients of the polynomials $\A,\B,\cC$
constructed in this article are unimodal.
\end{conj}

Let us mention that the solutions to~\eqref{PythEq} that we construct are
far from being the only existing solutions.
We will show that there are solutions to~\eqref{PythEq} 
satisfying the conditions \ref{Con1}, \ref{Con2}, \ref{Con3}, and the unimodality property
different from ours.
Their classification is a challenging problem.
It would also be interesting to understand what distinguishes 
the class of solutions related to $q$-rationals.

\section{Classical Pythagorean triples}

In this Section, we collect some simple and well-known facts about 
classical Pythagorean triples.
We attach great importance to the action of the modular group.

\subsection{$\SL(2,\Z)$-action}

Every Pythagorean triple $(a,b,c)$ can be identified with a symmetric
$2\times2$ matrix
\begin{equation}
\label{PythaM}
X_{(a,b,c)}=
\begin{pmatrix}
\frac{c+b}{2}&\frac{a}{2}\\[4pt]
\frac{a}{2}&\frac{c-b}{2}
\end{pmatrix}
\end{equation}
of rank~$1$, i.e. $\det(X_{(a,b,c)})=0$.
This is of course equivalent to the Pythagoras equation.
Moreover, the matrix $X_{(a,b,c)}$ has integer coefficients
if and only if $(a,b,c)$ is an integer multiple of a standard triple.
It is important to notice that $c$ is recovered from the above matrix as
the trace:
$$
c=\Tr(X_{(a,b,c)}).
$$
Clearly, $a$ and $b$ are also encoded by the matrix $X_{(a,b,c)}$,
but the trace is more fundamental and has an intrinsic meaning.

The interpretation of Pythagorean triples in the form of a matrix~\eqref{PythaM} 
allows one to define a natural action of
the group $\SL(2,\Z)$ of $2\times2$ unimodular matrices
$$
A=\begin{pmatrix}
\a&\b\\
\g&\d
\end{pmatrix},
\qquad\qquad
\a,\b,\g,\d\in\Z,
\quad
\a\d-\b\g=1
$$
on Pythagorean triples via
\begin{equation}
\label{LFAct}
X_{A(a,b,c)}:=
A\,X_{(a,b,c)}\,A^T,
\end{equation}
where $A^T$ is the matrix transposed to $A$.

The action~\eqref{LFAct} is transitive on the set of standard triples; see~\cite{Tra}.

\subsection{The Pythagorean tree}

The $\SL(2,\Z)$-action allows one to present the whole set of standard Pythagorean triples
in a form of a tree:
$$
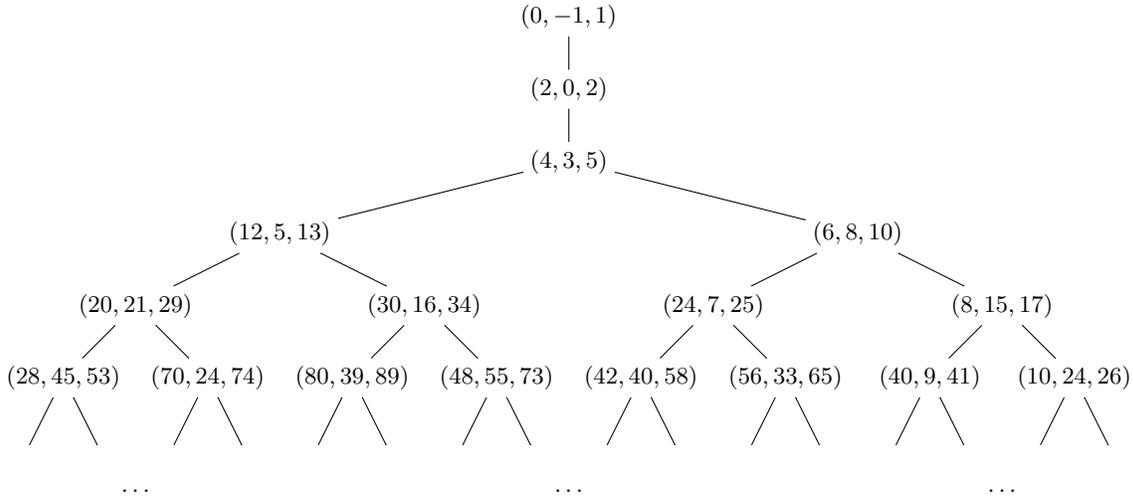
\begin{small}
\xymatrix @!0 @R=0.48cm @C=0.48cm
{
&&&&&&&&&&&&&&&&(0,-1,1)\ar@{-}[dd]&&\\
&&&&&&&&&&&&&&&&\\
&&&&&&&&&&&&&&&&(2,0,2)\ar@{-}[dd]&&\\
&&&&&&&&&&&&&&&&\\
&&&&&&&&&&&&&&&&(4,3,5)\ar@{-}[lllllllldd]\ar@{-}[rrrrrrrrdd]&&&&&&&&\\
&&&&&&&&&&&&&&&&\\
&&&&&&&&(12,5,13)\ar@{-}[lllldd]\ar@{-}[rrrrdd]
&&&&&&&&&&&&&&&&(6,8,10)\ar@{-}[lllldd]\ar@{-}[rrrrdd]&&&\\
&&&&&&&&&&&&&&&&&&&&&&&&\\
&&&&(20,21,29)\ar@{-}[lldd]\ar@{-}[rrdd]
&&&&&&&&(30,16,34)\ar@{-}[lldd]\ar@{-}[rrdd]
&&&&&&&&(24,7,25)\ar@{-}[lldd]\ar@{-}[rrdd]
&&&&&&&&(8,15,17)\ar@{-}[lldd]\ar@{-}[rrdd]\\
&&&&&&&&&&&&&&&&&&&&&&&&&&&&\\
&&(28,45,53)\ar@{-}[ldd]\ar@{-}[rdd]
&&&&(70,24,74)\ar@{-}[ldd]\ar@{-}[rdd]
&&&&(80,39,89)\ar@{-}[ldd]\ar@{-}[rdd]
&&&&(48,55,73)\ar@{-}[ldd]\ar@{-}[rdd]
&&&&(42,40,58)\ar@{-}[ldd]\ar@{-}[rdd]
&&&&(56,33,65)\ar@{-}[ldd]\ar@{-}[rdd]
&&&&(40,9,41)\ar@{-}[ldd]\ar@{-}[rdd]
&&&&(10,24,26)\ar@{-}[ldd]\ar@{-}[rdd]\\
&&&&&&&&&&&&&&&&&&&&&&
&&&&&&&&&&&\\
&&
&&&&
&&&&
&&&&
&&&&
&&&&
&&&&
&&&&&&\\
&&&&\ldots&&&&&&&&&&&&\ldots&&&&&&&&&&&&\ldots
}
\end{small}$$
Every right (resp. left) branch of the tree is obtained by the action~\eqref{LFAct} of the
standard generator~$R$ (resp.~$L$)
of $\SL(2,\Z)$, where
$$
R=
\begin{pmatrix}
1&1\\
0&1
\end{pmatrix},
\qquad\qquad
L=
\begin{pmatrix}
1&0\\
1&1
\end{pmatrix}.
$$
It is convenient to start the tree from the degenerate triple $(0,-1,1)$ corresponding to the matrix
\begin{equation}
\label{X0}
X_0=
\begin{pmatrix}
0&0\\
0&1
\end{pmatrix}.
\end{equation}
Consecutive actions of $R$ and $L$ on $(0,-1,1)$ 
lead to the triple $(4,3,5)$, and this is why the first two branches are presented as a vertical stem.

\begin{rem}
Another, perhaps more common, version of the Pythagorean tree that
contains only primitive Pythagorean triples can be obtained using the action
of the congruence subgroup $\Gamma(2)\subset\SL(2,\Z)$.
Our preference however is to keep the whole group $\SL(2,\Z)$,
which results in adding Pythagorean triples with $\gcd(a,b,c)=2$.
\end{rem}

\subsection{Euclide's formula in a matrix form}

In terms of the symmetric matrices~\eqref{PythaM}, the Euclide formula reads
$$
X_{(a,b,c)}=
\begin{pmatrix}
m^2&mn\\
mn&n^2
\end{pmatrix}=
\begin{pmatrix}
m\\
n
\end{pmatrix}
\begin{pmatrix}m, &n\end{pmatrix}.
$$
That is why it is practical to use another notation for $X_{(a,b,c)}$, namely
$X_{\frac{m}{n}}$.
The $\SL(2,\Z)$-action on pairs $(m,n)$ is then simply the linear action on $2$-vectors
$$
A:\begin{pmatrix}
m\\
n
\end{pmatrix}
\mapsto
A\begin{pmatrix}
m\\
n
\end{pmatrix},
$$
while the action on rational numbers is given by fractional-linear transformations
\begin{equation}
\label{LFT}
\begin{pmatrix}
\a&\b\\
\g&\d
\end{pmatrix}\left(x\right)=
\frac{\a x+\b}{\g x+\d}.
\end{equation}

Note that for $x\in\Q$ the result of~\eqref{LFT} can become infinite.
This is why the action~\eqref{LFT} is correctly understood as an action
on the rational projective line $\Q\cup\{\infty\}$,
where $\infty$ is represented by the quotient~$\frac{1}{0}$.
Note also that the center of  $\SL(2,\Z)$ acts trivially, therefore
is is more convenient to think of  $\PSL(2,\Z)$-action,
where $\PSL(2,\Z)$ is the quotient of $\SL(2,\Z)$ by the center: 
$$
\PSL(2,\Z)=\SL(2,\Z)/\Z_2.
$$
When thinking of $\PSL(2,\Z)$ instead of $\SL(2,\Z)$, one should
consider all $2\times2$ matrices up to a scalar multiple.
Note that the group $\PSL(2,\Z)$ is usually called the {\it modular group}.

\subsection{Relation to continued fractions}\label{CFSec}

Every rational number~$\frac{m}{n}>1$ can be written as a finite continued fraction
$$
\frac{m}{n}
\quad=\quad
a_1 + \cfrac{1}{a_2
          + \cfrac{1}{\ddots +\cfrac{1}{a_{k}} } },
          $$
with integer coefficients~$a_i$, such that $a_i\geq1$ for all $i\geq1$.
The standard notation is 
$$
\frac{n}{m}=[a_1,a_2,\ldots,a_{k}].
$$
The above continued fraction expansion is unique
if one chooses an even or odd number of coefficients.
For convenience, we assume that $k$ is odd.

Written in the matrix, or fractional-linear form, the above continued fraction reads
\begin{equation}
\label{MCF}
\frac{m}{n}=
R^{a_1}L^{a_2}R^{a_3}L^{a_4}\cdots{}R^{a_k}
\left(\frac{0}{1}\right).
\end{equation}
We conclude that the matrix~\eqref{PythaM} of a Pythagorean triple
corresponding to $\frac{n}{m}$ is given by
\begin{equation}
\label{Matmn}
X_{\frac{m}{n}}=
A\,X_0\,A^T,
\end{equation}
where $A=R^{a_1}L^{a_2}R^{a_3}L^{a_4}\cdots{}R^{a_k}$ and $X_0$ is as in~\eqref{X0}.

\section{A brief account on $q$-rationals}\label{RatSec}

The notion of $q$-deformed rationals was introduced in~\cite{MGOfmsigma}.
In this section, we briefly remind the definition and some properties
of $q$-rationals that will be useful for what follows.

\subsection{An intrinsic definition}

Given a rational number,~$\frac{m}{n}$, its $q$-deformation is a rational function in~$q$
$$
\left[\frac{m}{n}\right]_q=
\frac{\cN_{\frac{m}{n}}(q)}{\cD_{\frac{m}{n}}(q)},
$$
where the numerator $\cN_{\frac{m}{n}}$ and the denominator $\cD_{\frac{m}{n}}$
are polynomials in~$q$ that both depend on $m$ and $n$.
Assume that we know the $q$-analogue
of one point, for instance, 
$$
[0]_q:=0.
$$

More precisely, the $q$-rationals are characterized by the following two recurrence formulas
\begin{equation}
\label{RecRat}
\left[\frac{m}{n}+1\right]_q=
q\left[\frac{m}{n}\right]_q+1,
\qquad\qquad
\left[-\frac{n}{m}\right]_q=
-\frac{1}{q\left[\frac{m}{n}\right]_q}.
\end{equation}
Polynomials $\cN_{\frac{m}{n}}$ and $\cD_{\frac{m}{n}}$  
possess many interesting properties; see~\cite{MGOmn} and references cited therein.
In particular, they are monic  polynomials with positive integer coefficients.
An important property of unimodality was proved in~\cite{OgRa}.

\subsection{The $q$-deformed $\PSL(2,\Z)$-action}

The action is determined by two generators of $\PSL(2,\Z)$ represented by the matrices
\begin{equation}
\label{Gens}
R_{q}=
\begin{pmatrix}
q&1\\
0&1
\end{pmatrix},
\qquad\qquad
L_{q}=
\begin{pmatrix}
q&0\\
q&1
\end{pmatrix}
\end{equation}
considered up to a scalar multiple by a power of~$q$.
The map $\frac{m}{n}\mapsto\left[\frac{m}{n}\right]_q$ is characterized by the property
that it intertwines the $\PSL(2,\Z)$-action~\eqref{LFT} and the $\PSL(2,\Z)$-action
with the generators $R_q$ and $S_q$ as in~\eqref{Gens}.

\subsection{An explicit formula}

Given a rational number $\frac{m}{n}$ whose continued fraction expansion is
$\frac{m}{n}=[a_1,a_2,\ldots,a_k]$, where $k$ is odd.
A simple way to calculate a $q$-rational is to replace $R$ and $L$ in~\eqref{MCF} by
$R_q$ and $L_q$.
$$
\left[\frac{m}{n}\right]_q
=
R_q^{a_1}L_q^{a_2}R_q^{a_3}L_q^{a_4}\cdots{}R_q^{a_k}
\left(\frac{0}{1}\right).
$$
This can be viewed as one of several equivalent definitions:

\begin{ex}
The first interesting examples
$$
\left[\frac{5}{2}\right]_{q}=
\frac{1+2q+q^{2}+q^{3}}{1+q},
\qquad\qquad
\left[\frac{5}{3}\right]_{q}=
\frac{1+q+2q^{2}+q^{3}}{1+q+q^{2}}
$$
illustrates the fact that the numerator (and denominator) of the $q$-rational
$\frac{m}{n}$ depend both on~$m$ and~$n$.
Observe that the polynomials in  the numerators
of these $q$-rationals are precisely the polynomials $\cC$ and~$\cC^*$ 
from our first example.

\end{ex}

\subsection{The total positivity property}\label{TPSec}

An important property of the $q$-rationals is the fact that the $q$-deformation
``preserves the order'' of rationals 
in the following sense (see~\cite{MGOfmsigma}, Theorem~2).
Suppose we have two rationals~$\frac{m}{n}>\frac{m'}{n'}$, 
then the polynomial 
$$
\mathcal{P}_{\frac{n}{m},\frac{n'}{m'}}(q)=
\cN_{\frac{m}{n}}(q)\cD_{\frac{m'}{n'}}(q)-\cD_{\frac{m}{n}}(q)\cN_{\frac{m'}{n'}}(q)
$$
has positive integer coefficients.
This statement is topological in nature,
since every ordered set is endowed with a natural topology.
It was essential for extension of the notion of $q$-rationals to irrationals; 
see~\cite{MGOexp}.
This property will be important for us.

\subsection{The inverse $q$-rational}

The following property of $q$-rationals 
can be deduced from~\eqref{RecRat}; see~\cite{LMGadv}.
It allows one to calculate 
$q$-analogs for inverse rationals.
\begin{equation}
\label{InvqRat}
\left[\frac{n}{m}\right]_q=
\frac{1}{\left[\frac{m}{n}\right]_{q^{-1}}}.
\end{equation}
Using this formula it is easy to calculate explicitly the numerator and denominator.

\begin{cor}
\label{NDProp}
For every $\frac{m}{n}\in\Q$, one has
\begin{equation}
\label{NumDenExp}
\cN_{\frac{n}{m}}(q)=q^d\cD_{\frac{m}{n}}(q^{-1}),
\qquad\qquad
\cD_{\frac{n}{m}}(q)=q^d\cN_{\frac{m}{n}}(q^{-1}),
\end{equation}
where $d=\max(\deg(\cN_{\frac{m}{n}}),\deg(\cD_{\frac{m}{n}}))$.
\end{cor}

\begin{proof}
This readily follows from~\eqref{InvqRat}.
\end{proof}

\begin{rem}
Another formula was obtained in~\cite{Per2}:
$$
\left[\frac{n}{m}\right]_q=
\frac{\left(q-1\right)\left[\frac{m}{n}\right]_q+1}{q\left[\frac{m}{n}\right]_q+1-q}.
$$
One of the advantages of this formula is that it is not necessary to invert the parameter~$q$
and can be useful in many different situations.
\end{rem}

\section{A construction of $q$-Pythagorean triples}\label{SecTwo}

In this section, we describe a $q$-analogue for each
standard Pythagorean triple.
We show that our solutions satisfy the conditions~\ref{Con1}--\ref{Con3},
thus proving Theorem~\ref{ExtUniq}.

\subsection{The main definition}\label{DefSec}

Given a rational number~$\frac{m}{n}>1$ and its continued fraction
expansion $\frac{m}{n}=[a_1,a_2,\ldots,a_k]$, where $k$ is odd,
we define a $q$-analogue of the matrix~\eqref{PythaM} and~\eqref{Matmn}
by the following formula:
\begin{equation}
\label{Xq}
X_{\frac{m}{n}}(q):=
A_q\,X_0\,A_q^T,
\end{equation}
where 
$$
A_q:=R_q^{a_1}L_q^{a_2}R_q^{a_3}L_q^{a_4}\cdots{}R_q^{a_k}
\qquad\hbox{and}\qquad
A_q^T:=L_q^{a_k}R_q^{a_{k-1}}L_q^{a_{k-2}}R_q^{a_4}\cdots{}L_q^{a_1}.
$$

\begin{rem}
Note that $A_q^T$ is not exactly the transposition of~$A_q$.
Indeed, the notion of transposition in the $q$-deformed case has the following form
$$
\begin{pmatrix}
\a(q)&\b(q)\\[4pt]
\g(q)&\d(q)
\end{pmatrix}^T
:=\begin{pmatrix}
\a(q)&q^{-1}\g(q)\\[4pt]
q\b(q)&\d(q)
\end{pmatrix}.
$$
For details; see~\cite{XY}.
\end{rem}

\subsection{The trace of the matrix $X_{\frac{m}{n}}(q)$}\label{qEuclSec}

We have the following statement.

\begin{prop}
\label{TraceProp}
The trace of the matrix~\eqref{Xq} has the following expression in terms of the
corresponding $q$-rational
\begin{equation}
\label{TraceForm}
\Tr\left(X_{\frac{m}{n}}(q)\right)=
q\cN_{\frac{m}{n}}(q)^2+\cD_{\frac{m}{n}}(q)^2,
\end{equation}
where $\cN_{\frac{m}{n}}(q)$ and $\cD_{\frac{m}{n}}(q)$ are the numerator and denominator 
of the $q$-deformedrational~$\left[\frac{m}{n}\right]_q$.
\end{prop}

\begin{proof}
The matrixces $A_q$ and $A_q^T$ have the following explicit form
$$
A_q=
\begin{pmatrix}
*&\cN_{\frac{m}{n}}(q)\\[4pt]
*&\cD_{\frac{m}{n}}(q)
\end{pmatrix},
\qquad\qquad
A_q^T=
\begin{pmatrix}
*&*\\[4pt]
q\cN_{\frac{m}{n}}(q)&\cD_{\frac{m}{n}}(q)
\end{pmatrix}.
$$
Note that the first column of $A_q$ and the first row $A_q^T$ 
are not relevant, but can be calculated from the continued fraction
of~$\frac{m}{n}$; see~\cite{MGOfmsigma}.
It follows from~\eqref{Xq} that
\begin{eqnarray*}
X_{\frac{m}{n}}(q) &=&
\begin{pmatrix}
*&\cN_{\frac{m}{n}}(q)\\[4pt]
*&\cD_{\frac{m}{n}}(q)
\end{pmatrix}
\begin{pmatrix}
0&\;0\\[4pt]
0&\;1
\end{pmatrix}
\begin{pmatrix}
*&*\\[4pt]
q\cN_{\frac{m}{n}}(q)&\cD_{\frac{m}{n}}(q)
\end{pmatrix}\\[6pt]
&=&
\begin{pmatrix}
q\cN_{\frac{m}{n}}(q)^2&*\\[4pt]
*&\cD_{\frac{m}{n}}(q)^2
\end{pmatrix},
\end{eqnarray*}
and hence the result.
\end{proof}

\begin{defn}
For every rational $\frac{m}{n}$, we thus naturally associate the polynomial
$$
\cC_{\frac{m}{n}}(q):=q\cN_{\frac{m}{n}}(q)^2+\cD_{\frac{m}{n}}(q)^2.
$$
\end{defn}

Our goal is to show that the product of the polynomial $\cC_{\frac{m}{n}}(q)$ 
and its reciprocal $\cC^*_{\frac{m}{n}}(q)$ (see~\eqref{RecipEq}) satisfies the 
$q$-Pythagoras equation~\eqref{PythEq}.

\subsection{Solutions to the $q$-Pythagoras equation}\label{qPitSec}

We have the following result.

\begin{thm}
\label{CalcThm}
The polynomials
\begin{eqnarray}
\label{qEucl1}
\A_{\frac{m}{n}}(q)&=&
q\cN_{\frac{m}{n}}(q)\cN_{\frac{n}{m}}(q)+\cD_{\frac{m}{n}}(q)\cD_{\frac{n}{m}}(q),
\\[4pt]
\label{qEucl2}
\B_{\frac{m}{n}}(q)&=&
\cN_{\frac{m}{n}}(q)\cD_{\frac{n}{m}}(q)-\cD_{\frac{m}{n}}(q)\cN_{\frac{n}{m}}(q),
\\[4pt]
\label{qEucl3}
\cC_{\frac{m}{n}}(q)&=&
q\cN_{\frac{m}{n}}(q)^2+\cD_{\frac{m}{n}}(q)^2.
\end{eqnarray}
are solutions to~\eqref{PythEq}.
\end{thm}

\begin{proof}
From the definition of the reciprocal polynomial,
\begin{eqnarray*}
\cC^*_{\frac{m}{n}}(q) &=&
q^{\deg(\cC)}\cC_{\frac{m}{n}}(q^{-1})\\
&=&
q^{\deg(\cC)}\left(q^{-1}\cN_{\frac{m}{n}}(q^{-1})^2+\cD_{\frac{m}{n}}(q^{-1})^2\right).
\end{eqnarray*}
The degree of $\cC$ is given by
$$
\deg(\cC)=2d+1,
$$
where $d=\max(\deg(\cN_{\frac{m}{n}}),\deg(\cD_{\frac{m}{n}}))$, as in~\eqref{NumDenExp}.
Using~\eqref{NumDenExp}, we thus obtain
$$
\cC^*_{\frac{m}{n}}(q)=
\cD_{\frac{n}{m}}(q)^2+q\cN_{\frac{n}{m}}(q)^2.
$$
It follows that
$$
\cC_{\frac{m}{n}}(q)\cC^*_{\frac{m}{n}}(q)=
\left(q\cN_{\frac{m}{n}}(q)^2+\cD_{\frac{m}{n}}(q)^2\right)
\left(q\cN_{\frac{n}{m}}(q)^2+\cD_{\frac{n}{m}}(q)^2\right).
$$
Finally, using the classical Brahmagupta-Fibonacci square identity
$$
(x^2+y^2)(z^2+t^2)=(xz+yt)^2+(xt-yz)^2,
$$ 
for $x=q^\half\cN_{\frac{m}{n}}(q), 
y=\cD_{\frac{m}{n}}(q),
z=q^\half\cN_{\frac{n}{m}}$, and $t=\cD_{\frac{n}{m}}(q)$,
one obtains
\begin{eqnarray*}
\cC_{\frac{m}{n}}(q)\cC^*_{\frac{m}{n}}(q) &=&
\left(q\cN_{\frac{m}{n}}(q)\cN_{\frac{n}{m}}(q)+\cD_{\frac{m}{n}}(q)\cD_{\frac{n}{m}}(q)\right)^2\\
&&+q\left(\cN_{\frac{m}{n}}(q)\cD_{\frac{n}{m}}(q)-\cD_{\frac{m}{n}}(q)\cN_{\frac{n}{m}}(q)\right)^2,
\end{eqnarray*}
as desired.
\end{proof}

\begin{rem}
The expressions \eqref{qEucl1}--\eqref{qEucl3} constitute a $q$-analogue of the Euclide
formula~\eqref{EuclForm}.
\end{rem}

\subsection{Properties of the polynomials $\A_{\frac{m}{n}},\B_{\frac{m}{n}}$ and $\cC_{\frac{m}{n}}$}\label{PropSec}

Let us give some simple properties of the polynomials 
$\A_{\frac{m}{n}},\B_{\frac{m}{n}}$ and $\cC_{\frac{m}{n}}$ defined by
\eqref{qEucl1}--\eqref{qEucl3}.

\begin{prop}
\label{PropProp}
If $\frac{m}{n}\geq1$, then

(i)
The polynomials $\A_{\frac{m}{n}}$ and $\B_{\frac{m}{n}}$  are self-reciprocal.

(ii)
All three polynomials $\A_{\frac{m}{n}},\B_{\frac{m}{n}}$ and $\cC_{\frac{m}{n}}$
are monic and have positive coefficients.

\end{prop}

\begin{proof}
Part (i).
By~\eqref{NumDenExp}, another expression for 
the polynomials $\A_{\frac{m}{n}}$ and $\B_{\frac{m}{n}}$ is as follows
\begin{eqnarray*}
\A_{\frac{m}{n}}(q) &=&
q^d\left(
q\cN_{\frac{m}{n}}(q)\cD_{\frac{m}{n}}(q^{-1})
+\cD_{\frac{m}{n}}(q)\cN_{\frac{m}{n}}(q^{-1})
\right),
\\[6pt]
\B_{\frac{m}{n}}(q)&=&
q^d\left(\cN_{\frac{m}{n}}(q)\cN_{\frac{m}{n}}(q^{-1})
-\cD_{\frac{m}{n}}(q)\cD_{\frac{m}{n}}(q^{-1})\right),
\end{eqnarray*}
where $d=\deg(\cN_{\frac{m}{n}})>\deg(\cD_{\frac{m}{n}})$.
The above expression for $\B_{\frac{m}{n}}$ is manifestly self-reciprocal.
For $\A_{\frac{m}{n}}$ one has $\deg\left(\A_{\frac{m}{n}}\right)=2d+1$
and therefore
$$
\begin{array}{rcl}
\A^*_{\frac{m}{n}}(q) &=&
q^{\deg\left(\A_{\frac{m}{n}}\right)}\A_{\frac{m}{n}}(q^{-1})\\[6pt]
&=&
q^{d}\cN_{\frac{m}{n}}(q^{-1})\cD_{\frac{m}{n}}(q)
+q^{d+1}\cD_{\frac{m}{n}}(q^{-1})\cN_{\frac{m}{n}}(q)\\[6pt]
&=&\A_{\frac{m}{n}}(q).
\end{array}
$$

Part (ii).
For $\frac{m}{n}\geq1$ polynomials 
$\cD_{\frac{m}{n}}$ and $\cD_{\frac{n}{m}}$ are monic, and so is $\A_{\frac{m}{n}}$.
The polynomial $\cN_{\frac{m}{n}}$ is also monic, 
while the lower coefficient of $\cN_{\frac{n}{m}}$ vanishes.
Hence $\B_{\frac{m}{n}}$ is monic as well.

The polynomials $\A_{\frac{m}{n}}$ and $\cC_{\frac{m}{n}}$ have positive coefficients since
all numerators and denominators of the positive $q$-rationals have positive coefficients; 
see~\cite{MGOfmsigma}.
The fact that $\B_{\frac{m}{n}}$ has positive coefficients is a corollary of
Theorem~2 of~\cite{MGOfmsigma}.
\end{proof}

In all the examples we explored with numerical experimentation,
the formulas~\eqref{qEucl1}--\eqref{qEucl3} produce unimodal polynomials.
However, we were not able to prove the unimodality in general.

\subsection{Further examples}

To conclude this article, here are a few more examples.
We will consistently indicate whether the polynomials are factorable. 
A surprise awaits the persistent reader who makes it to the end.

(1)
Consider the Pythagorean triple $(12,5,13)$ that corresponds to the rational~$\frac{3}{2}$.
The $q$-analogues of this $q$-rational and its inverse are 
$$
\left[\frac{3}{2}\right]_q=\frac{1+q+q^2}{1+q},
\qquad\qquad
\left[\frac{2}{3}\right]_q=\frac{q+q^2}{1+q+q^2}.
$$
Applying~\eqref{qEucl1}--\eqref{qEucl3}, we get
\begin{eqnarray*}
\A_{\frac{3}{2}}(q)&=&
(1+q)(1+q^2)(1+q+q^2),
\\[4pt]
\B_{\frac{3}{2}}(q)&=&
1+q+q^2+q^3+q^4=[5]_q,
\\[4pt]
\cC_{\frac{3}{2}}(q)&=&
1+3q+3q^2+3q^3+2q^4+q^5.
\end{eqnarray*}

(2)
The Pythagorean triple $(56,33,65)$ corresponds to the rational~$\frac{7}{4}$.
In this case,
$$
\left[\frac{7}{4}\right]_q=\frac{1+q+2q^2+2q^3+q^4}{1+q+q^2+q^3},
\qquad\qquad
\left[\frac{4}{7}\right]_q=\frac{1+q+2q^2+2q^3+q^4}{1+q+q^2+q^3}
$$
In this case, applying~\eqref{qEucl1}-\eqref{qEucl3} we have
\begin{eqnarray*}
\A_{\frac{7}{4}}(q)&=&
(1+q)(1+q^2)(1+2q+3q^2+2q^3+3q^4+2q^5+q^6),
\\[4pt]
\B_{\frac{7}{4}}(q)&=&
(1+q+q^2)(1+q+2q^2+3q^3+2q^4+q^5+q^6),
\\[4pt]
\cC_{\frac{7}{4}}(q)&=&
1+3q+5q^2+9q^3+11q^4+12q^5+11q^6+8q^7+4q^8+q^9.
\end{eqnarray*}
It is interesting to notice that $\cC_{\frac{7}{4}}$ 
is not factorable and $\A_{\frac{7}{4}}$ is not a multiple of 
a polynomial corresponding to $7$.
Note also that all three polynomials, $\A_{\frac{7}{4}},\B_{\frac{7}{4}}$, 
and $\cC_{\frac{7}{4}}$, are unimodal.

(3) 
Take the simple example of the Pythagorean triple $(24,7,25)$
that corresponds to the rational $\frac{4}{3}$.
In this case, we found two different $q$-deformations satisfying all the conditions.
The first $q$-Pythagorean triple is obtained by applying~\eqref{qEucl1}-\eqref{qEucl3}:
\begin{eqnarray*}
\A_{\frac{4}{3}}(q) &=& (1+q)(1+q^2)^2(1+q+q^2),
\\[4pt]
\B_{\frac{4}{3}}(q) &=& 1+q+q^2+q^3+q^4+q^5+q^6=[7]_q,
\\[4pt]
\cC_{\frac{4}{3}}(q) &=& 1+3q+5q^2+5q^3+5q^4+3q^5+2q^6+q^7.
\end{eqnarray*}
But we also have another solution to~\eqref{PythEq}!
It looks completely different from the above one:
\begin{eqnarray*}
\A'_{\frac{4}{3}}(q) &=& (1+q)(1 + 10q + q^2),
\\[4pt]
\B'_{\frac{4}{3}}(q) &=& 1 + 5q + q^2,
\\[4pt]
\cC'_{\frac{4}{3}}(q) &=& 1 + 10q + 13q^2 +q^3.
\end{eqnarray*}

A complete classification of solutions to~\eqref{PythEq} is a challenging problem
which is out of reach so far.
Can we distinguish our solutions~\eqref{qEucl1}-\eqref{qEucl3} from the other classes? 
They might be the ``most quantized'' solutions, those with maximal number of terms.
Given that each term of a $q$-analogue should have a combinatorial meaning, 
this would be beneficial.
But this combinatorial meaning remains to be elucidated.

\bigskip

\noindent{\bf Acknowledgments}.
We are grateful to Perrine Jouteur for enlightening discussions.

\bigskip

\noindent
{\bf Corresponding Author}: Valentin Ovsienko.

\bigskip

\noindent
{\bf Funding declaration}: there was no funding for this manuscript.


\bigskip 

\bigskip  

\noindent
{\sc

\noindent
{Hugo Mathevet,
Universit\'e de Reims Champagne Ardenne,
Laboratoire de Math\'e\-ma\-tiques, CNRS UMR9008,
Moulin de la Housse - BP 1039,
51687 Reims cedex 2,
France,
}\\
{Email:  hugo.mathevet@univ-reims.fr}

\medskip

\noindent
{Sophie Morier-Genoud,
Universit\'e de Reims Champagne Ardenne,
Laboratoire de Math\'e\-ma\-tiques, CNRS UMR9008,
Moulin de la Housse - BP 1039,
51687 Reims cedex 2,
France,
}\\
{Email:  sophie.morier-genoud@univ-reims.fr}

\medskip

\noindent
{Valentin Ovsienko,
Centre National de la Recherche Scientifique,
Laboratoire de Math\'e\-ma\-tiques,
Universit\'e de Reims Champagne Ardenne,
Moulin de la Housse - BP 1039,
51687 Reims cedex 2,
France},\\
{Email:  valentin.ovsienko@univ-reims.fr}
}

\end{document}